\RequirePackage{fix-cm}
\documentclass[smallextended]{svjour3}       
\smartqed  
\usepackage{graphicx}
\usepackage{amsmath,amssymb,amsfonts,
latexsym, amscd, amsfonts, epsfig,
eepic,epic,psfrag,setspace,eufrak}
\usepackage{mathrsfs}
\usepackage{color}
\usepackage{eucal}
\usepackage{eufrak}
\usepackage[all]{xypic}
\usepackage{xspace}
\begin{document}

\title{Enumeration of linear chord diagrams}

\author{J.\ E.\ Andersen \and R.\ C.\  Penner \and C.\ M.\ Reidys \and M.\ S.
Waterman}

\institute{J.\ E.\ Andersen \at
Center for the Quantum Geometry of Moduli Spaces, Aarhus University, 
DK-8000 Aarhus C, Denmark \\
\email{andersen{\char'100}imf.au.dk}
\and
R.\ C.\  Penner \at
Center for the Quantum Geometry of Moduli Spaces, Aarhus University, 
DK-8000 Aarhus C, Denmark\\
and Departments of Math and Physics, Caltech, Pasadena, CA 91125 USA \\
\email{rpenner{\char'100}imf.au.dk}
\and
C.\ M.\ Reidys \at
Center for Combinatorics, LPMC-TJKLC, Nankai University, Tianjin
300071, P.R.~China\\
Tel.: +86-22-2350-6800\\
Fax:   +86-22-2350-9272 \\
\email{duck{\char'100}santafe.edu}
\and
M.\ S. Waterman \at
Departments of Biological Sciences, Mathematics, Computer Science, 
University of Southern California, Los Angeles, CA 90089, USA \\
\email{msw{\char'100}usc.edu}
}

\date{October, 2010}

\maketitle

\begin{abstract}
A linear chord diagram canonically determines a fatgraph and hence has an associated genus $g$.
We compute the natural generating function ${\bf C}_g(z)=\sum_{n\geq 0} {\bf c}_g(n)z^n$ for the number ${\bf c}_g(n)$
of linear chord diagrams of fixed genus $g\geq 1$ with a given number $n\geq 0$ of chords and find the remarkably simple formula
${\bf C}_g(z)=z^{2g}R_g(z) (1-4z)^{{1\over 2}-3g}$,
where $R_g(z)$ is a polynomial of degree at most
$g-1$ with integral coefficients satisfying $R_g({1\over 4})\neq 0$ and $R_g(0) = {\bf c}_g(2g)\neq 0.$ 
In particular, ${\bf C}_g(z)$ is algebraic over $\mathbb C(z)$,
which generalizes the corresponding classical fact for the generating function
${\bf C}_0(z)$ of the Catalan numbers.  As a corollary, we also calculate
a related generating function germaine to the enumeration of knotted
RNA secondary structures, which is again found to be algebraic.
\end{abstract}


\section{Introduction}\label{Intro}

A linear chord diagram consists of a line segment called its backbone to which are attached a number $n\geq 0$ of chords with distinct endpoints.
These  combinatorial structures occur in a number of instances  in pure mathematics including finite type invariants of knots and links
\cite{Barnatan95,Kontsevich93}, the representation theory of Lie algebras \cite{CSM}, the geometry of moduli spaces of flat connections on surfaces
\cite{AMR1,AMR2}, mapping class groups \cite{BAMP} and the  Four-Color Theorem \cite{Barnatan97},  and in applied mathematics including codifying the
pairings among nucleotides in RNA molecules \cite{Reidys11}, or more generally
the contacts of any binary macromolecule \cite{PKWA,Penner-Waterman,Waterman95},
and in the analysis of
data structures \cite{Flajolet80,FFV}.

As the title indicates, this paper is dedicated to enumerative problems associated with linear chord diagrams.  It is obvious that
there are $(2n-1)!!$ many distinct linear chord diagrams with $n$ chords. 
Partly because of their relevance to the number of finite type invariants of
knots, which corresponds to a quotient of the collection of linear chord
diagrams, sophisticated related enumerative problems have been studied in 
\cite{Flajolet-Noy,Klazar,Ng-S,Stoimenow98,Zagier01}.
In contrast to these, our approach depends upon a certain filtration of the
collection of all linear chord diagrams as follows.

\begin{figure}[ht]
\centerline{\epsfig{file=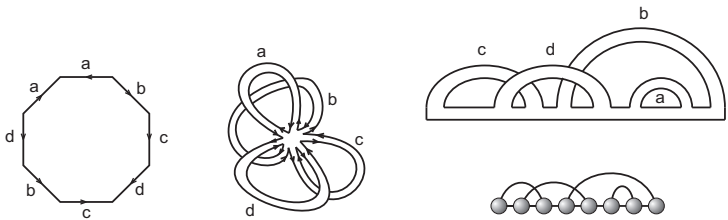,width=0.8\textwidth} \hskip8pt}
\caption{\small Equivalence between a pairwise identification of the
sides of an octagon (left), the surface $F({\mathbb{G}})$ where all
vertices of the backbone are collapsed into a single vertex (note we
therefore have $2-2g-r=1-n$) (middle) and the surface
$F({\mathbb{G}})$ induced by a linear chord diagram $G$ (right). } \label{F:collaps}
\end{figure}

Drawing a picture of a linear chord diagram $G$ in the plane with its backbone lying in the real axis and the chords in the upper half-plane
determines a cyclic ordering on the half-edges of the underlying graph incident on each vertex, thus defining
a corresponding ``fatgraph" $\mathbb G$ to which is canonically associated a topological surface
$F(\mathbb G)$ (cf. $\S$\ref{S:fatgraphs}) of some
genus; see Figure \ref{F:collaps}.  It follows that to each linear chord diagram $G$ is naturally associated the corresponding genus
of $F(\mathbb G)$,
and we let ${\bf c}_g(n)$ denote the number of distinct linear chord diagrams with $n$ chord of genus $g$ with corresponding
generating function ${\bf C}_g(z)=\sum_{n\geq 0} {\bf c}_g(n)z^n$, for each $g\geq 0$.

In particular, the Catalan numbers ${\bf c}_0(n)$, i.e., the number of triangulations of a polygon with $n+2$ sides,  also
enumerate linear chord diagrams of genus zero.  Their
recursion is evidently given by ${\bf c}_0(n+1)=\sum_{i=0}^n{\bf c}_0(i){\bf  c}_0(n-i)$ with basis ${\bf c}_0(0)=1$, which
implies ${\bf C}_0(z)=\sum_{n\geq 0} {\bf c}_0(n) z^n=1+z[{\bf C}_0(z)]^2$, whence
$${\bf C}_0(z)={{1-\sqrt{1-4z}}\over{2z}}={{2}\over{1+\sqrt{1-4z}}},$$
and so
$${\bf c}_0(n)=\binom{2n}{n} {1\over {n+1}}={{(2n)!}\over {(n+1)!n!}}$$
since $$\sqrt{1+z}=1-2\sum_{n\geq 1}\binom{2n-2}{n-1}~({{-1}\over {\hskip 1.3ex 4}})^n{{z^n}\over n}.$$

In fact, these numbers ${\bf c}_g(n)$ had been computed in another generating function over two decades ago
by Harer-Zagier \cite{Harer-Zagier} in the equivalent guise of  the number of side pairings of a polygon with $2n$ sides
that produce a surface of genus $g$, cf. Figure \ref{F:collaps}, namely,
$$1+2\sum_{n\geq 0} \sum_{2g\leq n} {{{\bf c}_g(n)N^{n+1-2g}}\over {(2n-1)!!}}~z^{n+1}=\biggl ({{1+z}\over {1-z}}\biggr )^N,$$
a striking and beautiful formula which is the starting point for our computations and for which we therefore
also provide a novel proof in Lemma \ref{L:claim1} depending only on the character theory of the symmetric group.
Indeed, this formula was a crucial intermediate step for the calculation of the virtual Euler characteristic
of Riemann's moduli space given in \cite{Harer-Zagier} and independently
in \cite{Penner88} using a novel matrix model.

The topological filtration of linear chord diagrams by genus discussed here was also considered in \cite{Cori-Marcus}, where enumerations
for genus one and for maximal genus were obtained.  Furthermore, there is physics literature initiated in \cite{Orlandetal02}
on RNA enumeration based on matrix models which relies on the genus of linear chord diagrams including
the derivation of another expression for the ${\bf c}_g(n)$ in terms of Laguerre polynomials \cite{Vernizzietal05} and the comparison of expected with observed genera \cite{Bonetal08,Vernizzietal06}.

Though the numbers ${\bf c}_g(n)$ have thus been known in various forms for some time, the natural generating functions
${\bf C}_g(z)$ have not been computed heretofore, and it is this that we accomplish here.  In fact, we shall prove in Theorem
\ref{E:GF} that
for any $g\ge 1$
\begin{eqnarray*}
\mathbf{C}_g(z) = \, P_g(z)\frac{\sqrt {1-4\,z}}{(1-4z)^{3g}},
\end{eqnarray*}
where $P_g(z)$ is a polynomial defined over the integers of degree at most $(3g-1)$ that is divisible by $z^{2g}$ with
$P_g(1/4)\neq 0$.  In particular and surprisingly, $\mathbf{C}_g(z)$ is algebraic over $\mathbb C(z)$ for all $g\geq 1$ just as is
the Catalan generating function $\mathbf{C}_0(z)$, which our results thus generalize.

In fact, the polynomials $P_g(z)$ empirically have degree exactly $3g-1$ and all their coefficients are positive.  These attributes of positivity and exact degree have yet to be proved however.  The first several such polynomials are given below.
\begin{eqnarray*}
P_1(z) &=& z^2,\\
P_2(z) &=& 21z^4\, \left( z+1 \right)\\
P_3(z) &=&  11z^6\, \left( 158\,{z}^{2}+558\,z+135 \right),\\
P_4(z) &=&143z^8\left( 2339\,{z}^{3}+18378\,{z}^{2}+13689\,z+1575 \right),\\
P_5(z) &=&  88179z^{10}\, \left( 1354\,{z}^{4}+18908\,
{z}^{3}+28764\,{z}^{2}+9660\,z+675 \right).
\end{eqnarray*}
In light of these properties, it is natural to speculate that the polynomials 
$P_g(z)$ themselves solve an enumerative problem, and it is interesting in 
this vein to compare coefficients for $P_2$ and $P_3$ with fatgraph tables 
such as those in \cite{Milgram-Penner}.
   One might further speculate that there may be a purely combinatorial 
topological proof of Theorem \ref{E:GF} based on a construction that in 
some way ``inflates" these structures using genus zero diagrams, to wit
$$\aligned
{\bf C}_g(z)&= P_g(z) (\sqrt{1-4z})^{1-6g}\\
&=P_g(z)\biggl ( {{{\bf C}_0(z)}\over{2-{\bf C}_0(z)}}\biggr )^{6g-1}\\
&=P_g(z) \bigl ( {\bf C}_0(z)\bigr )^{6g-1}\biggl (1+({\bf C}_0(z)-1 
\bigr )+({\bf C}_0(z)-1 \bigr )^2+\cdots \biggr )^{6g-1}\\
&=P_g(z) \bigl ( {\bf C}_0(z)\bigr )^{6g-1}\biggl 
(1+z({\bf C}_0(z))^2 +z^2({\bf C}_0(z))^4+\cdots \biggr )^{6g-1}.\\
\endaligned$$

Motivated by enumerative problems for RNA, we also study a further class of combinatorial objects as follows.
A ``partial linear chord diagram" is a collection of chords attached to a backbone with $n$ vertices, where we now drop
the condition for a linear chord diagram that each vertex has an incident chord.  Furthermore,
two distinct chords with respective endpoints $i_1<j_1$ and $i_2<j_2$ are ``consecutively parallel" if $i_1=i_2-1\leq j_2=j_1-1$, consecutive parallelism generates the equivalence relation of  ``parallelism" whose equivalence classes are called {``stacks"}.
A chord connecting vertices which are consecutive along the backbone is called 
a ``1-chord".

A ``macromolecular diagram of minimum helix length $\sigma\geq 1$ on $n\geq 0$ vertices" is a partial linear chord diagram on $n$ vertices with no 1-chords so that each stack contains at least $\sigma$ edges; let ${\bf d}_{g,\sigma}(n)$ be the number of
all such partial linear chord diagrams of genus $g$ with generating function
${\bf D}_{g,\sigma}(z)=\sum_{n\geq 0 } {\bf d}_{g,\sigma}(n) z^n$.  We compute this generating function for $g,\sigma\geq 1$ in Theorem
\ref{T:genus} to be
\begin{eqnarray*}
{\bf D}_{g,\sigma}(z) & = & \frac{1}{u_\sigma(z) z^2-z+1}\
                            {\bf C}_g\left(\frac{u_\sigma(z)z^2}
                            {\left(u_\sigma(z) z^{2}-z+1\right)^2}\right),
\end{eqnarray*}
where $u_\sigma(z)
=\frac{(z^2)^{\sigma-1}}{z^{2\sigma}-z^2+1}$.
In particular, ${\bf D}_{g,\sigma}(z)$ is also algebraic over ${\mathbb C}(z)$, and
for arbitrary but fixed $g$ and $\gamma_2\approx 1.9685$, we have
\begin{equation}
{\bf d}_{g,2}(n)\sim k_g\,n^{3(g-\frac{1}{2})} \gamma_2^n,
\end{equation}
for some constant $k_g$.

The exponential growth rate of 1.9685 shows that the number of 
macromolecular diagrams grows much more slowly than the number of RNA
sequences over the natural alphabet. This implies the existence of
neutral networks \cite{Kimura,Reidys97a,Reidys02}, i.e.,
vast extended sets of RNA sequences all
folding into a single macromolecular structure. These neutral networks
are of key importance in the context of neutral evolution of RNA sequences.


\section{Background and Notation}\label{Back}

We formulate the basic terminology and notation for graphs and linear chord diagrams,
establish notions and notations for the symmetric group, recall
the fundamental ideas and constructions for fatgraphs, and finally combine
these ingredients for application in subsequent sections.

\subsection{Graphs and linear chord diagrams}

Let $G$ be a finite {\it graph} in the usual sense of the term comprised of vertices $V(G)$ and edges $E(G)$, where edges do not contain their endpoints and are not necessarily uniquely determined by them;  in other words, $G$ is a finite one-dimensional CW complex.   Removing a single point from an edge produces two components, each of which is called a {\it half-edge} of $G$.
A half-edge which contains a vertex $v$ in its closure is said to be {\it incident} on $v$,
and the number of distinct half-edges incident on $v$ is its {\it valence}.

Let $B=B_n$ denote the closed interval $[1,n]$ of real numbers between $1$ and $n\geq 2$ regarded as a graph with
$V(B)=B\cap\mathbb Z$ and $E(B)=\{(i,i+1):i=1,\ldots ,n-1\}$.  We shall refer
to the least and greatest elements of $V(B)$ as {\it extreme vertices}, which are univalent in $B$, and to the other vertices of $B$ as {\it interior vertices},
which are bivalent in $B$.
A {\it partial linear chord diagram $C$ on $n\geq 2$ vertices} is a graph containing $B$ so that $V(C)=V(B)$, each interior vertex of $C$ has valence at most 3, and the extreme vertices have valence at most 2.  $B$ is called the {\it backbone} of $C$, edges in $E(B)\subseteq E(C)$ are called {\it backbone edges}, and edges in the complement $E(C)-E(B)$ are called {\it chords}.

In particular, $C$ is called a {\it linear chord diagram} if every interior vertex has valence exactly three
and the extreme vertices have valence exactly two;  in particular, the number of vertices for a linear chord diagram is necessarily even.

Two distinct chords $e_1,e_2\in E(C)-E(B)$ in a (partial) linear chord diagram $C$ with respective endpoints $i_1<j_1$ and $i_2<j_2$ are {\it consecutively parallel} if $i_1=i_2-1\leq j_2=j_1-1$.  Consecutive parallelism generates the equivalence relation of {\it parallelism} on $E(C)$, and equivalence classes are called {\it stacks}.  A linear chord diagram in which every stack has cardinality at most one is called a {\it shape}.

Let  $\mathscr P(n)$ denote the collection of all partial linear chord diagrams on $n$ vertices and
${\mathscr C}(n), {\mathscr S}(n)$  the collections of
all linear chord diagrams and shapes on $2n$ vertices, respectively, so in particular, we have the inclusions $\mathscr P(2n)\supseteq \mathscr C(n)\supseteq \mathscr S(n)$.  There is a natural projection
$$\vartheta: \sqcup_{n\geq 1} {\mathscr P}(n)\to\sqcup_{n\geq 1} {\mathscr S}(n)$$
defined by contracting to a point any backbone edge at least one of whose interior endpoints is not trivalent or extreme endpoints is not bivalent as well as collapsing each non-empty stack onto a single chord, i.e.,  all least vertices  of  chords in a stack
are collapsed to a single vertex and likewise all greatest vertices.

Define a {\it $1$-chord} in a (partial) linear chord diagram to be a chord connecting two consecutive vertices $i$ and $i+1$ in the backbone.

1-chords are typically proscribed in the partial linear chord diagrams
that arise in applications to RNA owing to tensile rigidity of the backbone.  Furthermore, stacks of small cardinality are typically energetically unfavorable,
and one introduces a parameter $\sigma\geq 1$ specifying the minimum allowed cardinality of a stack.
A {\it macromolecular diagram of minimum helix length $\sigma\geq 1$ on $n\geq 0$ vertices} is a partial linear chord diagram on $n$ vertices with no 1-chords so that each stack contains at least $\sigma$ edges; let $\mathscr D_\sigma(n)$ denote the collection of all such partial linear chord diagrams. 

Our main results solve enumerative problems for linear chord diagrams, shapes, and macromolecular diagrams of fixed minimum helix length.

\subsection{Permutations}

The symmetric group $S_{2n}$ of all permutations on $2n$ objects will play a key role in our calculations.
We shall adopt the standard notation writing $(i_1,i_2,\ldots ,i_k)$ for the cyclic permutation $i_1\mapsto i_2\mapsto\cdots\mapsto i_k\mapsto i_1$ on distinct objects $i_1,\ldots , i_k$ and shall
compose permutations $\pi,\tau$ from right to left, so that $\pi\circ\tau(i)=\pi(\tau(i))$.
An {\it involution} $\iota $ is a permutation so that $\iota\circ\iota$ is the identity.

The conjugacy class of $\pi\in S_{2n}$ will be denoted $[\pi]$.  Conjugacy classes in $S_{2n}$ are identified with classes of partitions of $\{1,\ldots ,2n\}$, where in the standard slightly abusive notation,
$\pi\in[\pi ]=[1^{\pi_1} \cdots {2n}^{\pi_{2n}}]$ denotes a partition comprised of $\pi_k\geq 0$ many parts of size $k$, for $k=1,\ldots ,2n$,
i.e., $\pi$ is comprised of $\pi_k$ many $k$ cycles of pairwise disjoint supports, for $k=1,\ldots ,2n$, so
necessarily $\sum_{k=1}^{2n} k\pi_k=2n$.  A permutation $\pi$ is an involution if and only if
$\pi\in[1^{\pi_1}2^{\pi_2}]$, and it is fixed point free if and only if  $\pi_1=0$.  The number of elements in the class $[\pi]$ is given by
$$|\pi|= |[\pi]|= {{2n!}\over{\prod_{k=1}^{2n} k^{\pi_k}\pi_k!}}.$$

The irreducible characters $\chi^Y$ of $S_n$ are labeled by Young tableaux $Y$.  See
\cite{Sagan} for further details and background.

\subsection{Fatgraphs}\label{S:fatgraphs}

A {\it fatgraph} $\mathbb G$ is a graph $G$ together with the specification of a collection of cyclic orderings, called the fattening,
one such cyclic ordering on the half-edges incident on $v$ for each $v\in V(G)$.

A fatgraph $\mathbb G$ uniquely determines an oriented surface $F({\mathbb G})$ with boundary as follows.
For each $v\in V(G)$, consider an oriented surface isomorphic to a polygon $P_v$ of $2k$ sides containing in its interior a
single vertex of valence $k$ each of whose incident edges are also incident on a univalent vertex contained in alternating sides of $P_v$, which are identified
with the incident half-edges in the natural way
so that the induced counter-clockwise cyclic ordering on the boundary of $P_v$ agrees with the fattening of $\mathbb G$ about $v$.
The surface $F(\mathbb G)$ is the quotient of the disjoint union $\sqcup_{v\in V(G)} P_v$, where the frontier edges,
which are oriented with the polygons on their left, are identified by an orientation-reversing homeomorphism if the corresponding half-edges lie in a common edge of $G$.   This defines the oriented surface $F(\mathbb G)$, which is connected if $G$ is and has some associated genus $g(\mathbb G)\geq 0$ and number $r(\mathbb G)\geq 1$ of boundary components.

The various trees in the polygons $P_v$, for $v\in V(G)$, combine to
give a graph identified with $G$ embedded in $F(\mathbb G)$, so that we regard $G\subseteq F(\mathbb G)$.
In fact, $G$ is a deformation retraction of the surface $F(\mathbb G)$ by construction, so their  Euler characteristics agree, namely,
$$\chi(G)=\#V(G)-\#E(G)=2-2g(\mathbb G)-r(\mathbb G)=\chi (F(\mathbb G))$$
provided $G$ is connected, where $\#$ denotes cardinality.

A fatgraph $\mathbb G$ is uniquely determined by a pair of permutations on the half-edges of its underlying graph $G$ as follows.
Let $v_k\geq 0$ denote the number of $k$-valent vertices, for each
$k\geq 1,\ldots ,K$, where $K$ is the maximum valence of vertices of $G$ and
$$\sum_{k\geq 1} ^K kv_k=2\# E(G)=2n$$
is the number of half-edges.
The valencies of vertices of $G$ are thus succinctly described by a
permutation in the conjugacy class $[1^{v_1}2^{v_2}\cdots K^{v_K}]$.

In order to explicitly determine a permutation in this class,
specify a linear order on $V(G)$ as well as a distinguished half-edge incident on each $v\in V(G)$.
This determines a unique linear ordering on the half-edges of $G$ which restricts to the
the fattening at each vertex, where the distinguished half-edge is least,
so that one half-edge furthermore precedes another if it is incident on a preceding vertex.
There is thus a well-defined permutation $\tau\in [1^{v_1}2^{v_2}\cdots K^{v_K}]\subseteq S_{2n}$
whose disjoint cycles correspond to the fattenings at each vertex.

The second permutation $\iota\in S_{2n}$ is the product
$\iota=\prod_{i=1}^n (h,h')$
over all edges $e\in E(G)$
of disjoint transpositions $(h,h')$, where the distinct half-edges
$h,h'$ lie in the common edge $e$.  Thus, whereas the permutation $\tau$ determines the
valencies of vertices, the fixed-point free involution $\iota\in [2^n]$ determines the edges of $G$.

Several basic facts follow from this representation of a fatgraph $\mathbb G$ as a pair $\tau,\iota\in S_{2n}$
of permutations.
One important point that is easy to confirm is that the boundary components of $F(\mathbb G)$ are in one-to-one correspondence
with the cycles of $\tau\circ\iota$, i.e.,
$$\aligned
r(\mathbb G)&=~{\rm the~number~of~disjoint~cycles~comprising}~\tau\circ\iota\\
&=(\tau\circ\iota)_{1}+(\tau\circ\iota)_{2}+\cdots +(\tau\circ\iota)_{2n}.\\
\endaligned$$

Furthermore, isomorphism classes of fatgraphs with vertex valencies
$(v_k)_{k=1}^K$ are evidently in bijection with conjugacy classes of pairs $\tau,\iota\in S_{2n}$,
where $\tau\in [1^{v_1}2^{v_2}\cdots K^{v_K}]$ and $\iota\in [2^n]$.  In particular,
as a data type on the computer, fatgraphs are easily stored and manipulated as pairs
of permutations, and various enumerative problems can be formulated
in terms of Young tableaux.

See \cite{Penner88,PKWA} for more details on fatgraphs
and \cite{BIZ,Harer-Zagier,Milgram-Penner,Penner92,Itzykson-Zuber}
for examples of fatgraph enumerative
problems in terms of character theory for the symmetric groups.

\subsection{Fatgraphs and linear chord diagrams}

A regular planar projection of
a graph in 3-space determines a corresponding fattening on it, namely, the
counter-clockwise cyclic ordering in the plane of projection.  The crossings of edges
in the plane of projection can be arbitrarily resolved into under/over crossings
without affecting the resulting isomorphism class.  Furthermore, a band about each edge can
be added to
a neighborhood of the vertex set in the plane of projection respecting orientations in order to give an explicit
picture of the associated surface embedded in 3-space.  An example with
a single 8-valent vertex is illustrated in the middle of Figure \ref{F:collaps}; this fatgraph $\mathbb G$
can
be described by the pair $\tau=(1,2,3,4,5,6,7,8)$, $\iota=(1,5)(2,3)(4,7)(6,8)$, and indeed,
the cycles of $\tau\circ\iota=(1,6)(2,4,8,7,5)(3)$ correspond to the boundary components of $F(\mathbb G)$,
which has Euler characteristic -3, $r(\mathbb G)=3$, and $g(\mathbb G)=1$.

The standard planar representation of a (partial) linear chord diagram $C$ represents the backbone as a
real interval and non-backbone edges as semi-circles in the upper half plane as in the example on the bottom-right
in Figure \ref{F:collaps}. This  planar projection thus implicitly determines
the {\it canonical fattening} $\mathbb C$ of $C$ as above.
An example is given on the top-right in the figure
with corresponding permutations given by $\tau=(1,2)(3,4,5)(6,7,8)\cdots (18,19,20)(21,22)$, $\iota=(2,7)(4,13)(10,21)(16,19)$, where vertices are ordered left to right and rightmost backbone half-edges are distinguished to determine the linear ordering on half-edges.

Given a (partial) linear chord diagram $C$ with its corresponding fatgraph $\mathbb C$, consider the graph $G$ arising from
$C$ by collapsing its backbone to a single vertex together with its fattening $\mathbb G$ induced from
$\mathbb C$ in the natural way; for example, the fatgraph $\mathbb G$ in the middle of Figure \ref{F:collaps} arises in
this manner from the fatgraph $\mathbb C$ on the top-right.  We claim that the surfaces $F(\mathbb C)$ and $F(\mathbb G)$ have the same genus and number of boundary components, and indeed, it follows by construction that the two surfaces are homeomorphic and hence have the same invariants.

In particular, a linear chord diagram $C$ on $2n$ vertices gives rise to a fatgraph $\mathbb G_C$ with a single
vertex of valence of $2n$ and a distinguished half-edge (namely, the one coming just after the location
of the collapsed backbone), i.e., a pair of permutations $\tau\in[2n]$, $\iota\in[2^n]$. We may
thus define the number of boundary components and genus of a linear chord diagram
$$\aligned
r(C)&=r(\mathbb G_C)={\rm the~number~of~disjoint~cycles~comprising}~\tau\circ\iota,\\
g(C)&=g(\mathbb G_C)= {1\over 2}\bigl(n+1-r(C)\bigr).\\
\endaligned$$
Conversely, the specification of a pair $\tau\in[2n]$, $\iota\in[2^n]$ uniquely determines
a linear chord diagram on $2n$ vertices.  Equivalently, the enumeration of pairs $\tau\in[2n]$, $\iota\in[2^n]$ corresponds
to all possible edge-pairings of a polygon with $2n$ sides, as illustrated for the ongoing
example on the left in Figure \ref{F:collaps}.
Summarizing, we have the following:

\begin{proposition}\label{isoprop}
The following four sets are in bijective correspondence
$$\aligned
&\mathscr C(n)=\{ {\rm chord~diagrams~on}~2n~{\rm vertices}\},\\
&\{{\rm univalent~fatgraphs~with}~n~{\rm edges~and~a~distinguished~half-edge}\},\\
&\{{\rm edge-pairings~on~a~polygon~with}~2n~{\rm labeled~sides}\},\\
&\{{\rm pairs}~\tau,\iota\in S_{2n}:\tau\in[2n]~{\rm and}~\iota\in[2^n]\}.\\
\endaligned$$
\end{proposition}

Our first counting results will rely upon the bijection between $\mathscr C(n)$ and pairs of permutations in $S_{2n}$
established here.
The subsequent more refined results are tailored to
the macromolecular diagrams of interest in computational biology.

\subsection{Topological Filtrations and Generating Functions}

Let $\mathscr C_g(n)\supseteq \mathscr S_g(n)$ denote the collections of all linear chord diagrams and shapes of genus
$g\geq 0$ on $2n\geq 0$ vertices with respective generating functions
$$\aligned
{\bf C}_g(z)&=\sum _{n\geq 0}{\bf c}_g(n)z^n,\\
{\bf S}_g(z)&=\sum_{n\geq 0}{\bf s}_g(n)z^n,\\
\endaligned$$
where ${\bf c}_g(n)={\bf s}_g(n)=0$ if $2g>n$ since $r=n+1-2g$ has no positive solution $r>0$ unless $2g\leq n$.

Likewise, let $\mathscr C_g(n,m)\supseteq \mathscr S_g(n,m)$ denote the collections of all
linear chord diagrams and shapes of genus $g\geq 0$ on $2n\geq 0$ vertices containing $m\geq 0$ 1-chords with respective
generating functions
$$\aligned
{\bf C}_g(x,y)&=\sum_{m,n\geq 0}{\bf c}_g(n,m)x^ny^m,\\
{\bf S}_g(x,y)&=\sum_{m,n\geq 0}{\bf s}_g(n,m)x^ny^m,\\
\endaligned$$
where ${\bf c}_g(n,m)={\bf s}_g(n,m)=0$ if $2g>n$ or if $m>n$.

Notice that  the projection $\vartheta$ restricts to a surjection
$$\vartheta: \sqcup_{n\geq 0} {\mathscr C}_g(n,m)\to\sqcup_{n\geq 0} {\mathscr S}_g(n,m)$$
which collapses each stack to a chord and therefore evidently preserves both the genus $g$ and the number
$m$ of 1-chords.
For any shape $\gamma\in\sqcup_{n\geq 0}\mathscr S(n,m)$,  let
$$
\mathscr C_\gamma(n,m)=\mathscr C(n,m)\cap\vartheta^{-1}(\gamma)
$$ denote
the intersection with the fiber $\vartheta ^{-1}(\gamma)$ with its generating function
$
{\bf C}_\gamma (x,y)
$.

Turning finally to macromolecular diagrams, let $\mathscr D_{g,\sigma}(n)$ denote the subset of $\mathscr D_\sigma (n)$
comprised of diagrams with genus $g$ and let
$${\bf D}_{g,\sigma}(z)=\sum_{n\ge 0}{\bf d}_{g,\sigma}(n)z^n$$
denote the corresponding generating function.  Again, the projection $\vartheta$ restricts to a surjection
$$\vartheta: \sqcup_{n\geq 0} {\mathscr D}_{g,\sigma}(n)\to\sqcup_{n\geq 0} {\mathscr S}_g(n),$$
which preserves the genus.  For any shape  $\gamma\in\sqcup_{n\geq 0}\mathscr S(n)$,
let
$$
\mathscr D_{\gamma,\sigma}(n)=\mathscr D_{\sigma}(n)\cap\vartheta^{-1}(\gamma)
$$ denote
the intersection with the fiber $\vartheta ^{-1}(\gamma)$ with its generating function
$
{\bf D}_{\gamma,\sigma} (z)
$.

As a general notational point for any power series $P(z)=\sum a_iz^i$, we shall write $[z^i]P(z)=a_i$ for the  extraction of the coefficient $a_i$ of $z^i$.


\section{The generating function of linear chord diagrams of genus $g$}


We introduce the polynomial
\begin{equation}
P(n,x)=\sum_{\{g\mid  2g\le n\}} \mathbf{c}_g(n) \cdot x^{n+1-2g},
\end{equation}
which plays a key role in our
computation of ${\bf C}_g(z)$.
\begin{lemma}\cite{Harer-Zagier}\label{L:claim1}
Letting $x,z$ denote indeterminates, we have
\begin{eqnarray}
1+ 2\sum_{n=0}^\infty \frac{P(n,x)}{(2n-1)!!}\,
z^{n+1} & = & \left(\frac{1+z}{1-z}\right)^x.
\end{eqnarray}
\end{lemma}
\begin{proof}
Since the number of boundary components $r=n+1-2g$
of the fatgraph corresponding via Proposition \ref{isoprop} to a pair $\tau,\iota$ of permutations
equals the number of
disjoint cycles comprising $\tau\circ\iota$, we compute

$$\aligned
P(n,N)  =  \sum_{\{g\mid  2g\le n\}} \mathbf{c}_g(n) \, N^{n+1-2g}
  &=  \sum_{\iota\in [2^n]} N^{\sum_i(\tau\iota)_i}\\
  &=   \sum_{[\pi]} N^{\sum_i\pi_i}
\sum_{\iota\in [2^n]\atop \tau\iota\in [\pi]} 1\\
&= \sum_{[\pi]} N^{\sum_i\pi_i} \sum_{\sigma\in S_{2n}}
\delta_{[\sigma],[2^n]}\cdot\delta_{[\tau\sigma],[\pi]},\\
\endaligned$$
where $\tau\in [2n]$ is any fixed permutation and $\delta$ denotes the Kronecker delta function.

We claim that
\begin{equation}\label{firstclaim1}
\sum_{\sigma\in S_{2n}}
\delta_{[\sigma],[2^n]}\cdot\delta_{[\tau\sigma],[\pi]}=
\frac{(2n-1)!!}{\prod_jj^{\pi_j}\cdot \pi_j!}
\sum_{Y} \frac{\chi^Y([2^n])\chi^Y(\pi)\chi^Y([2n])}
{\chi^Y([1^{2n}])}.
\end{equation}
To this end, the orthogonality relations
\begin{eqnarray*}
\sum_Y \chi^Y(\sigma_1)\chi^Y(\sigma_2) & = &
\frac{(2n)!}{\vert [\sigma_1]\vert}\cdot
\delta_{[\sigma_1],[\sigma_2]},
\end{eqnarray*}
of the second kind give
\begin{eqnarray*}
\sum_{\sigma\in S_{2n}}
\delta_{[\sigma],[2^n]}\cdot\delta_{[\tau\sigma],[\pi]} & = &
\sum_{\sigma\in S_{2n}}\left[\frac{\vert [2^n]\vert}{(2n)!}
\sum_Y \chi^Y(\sigma)\chi^Y([2^n])\right]
\cdot \left[ \frac{\vert [\pi]\vert}{(2n)!}
\sum_{Y'} \chi^{Y'}(\tau\sigma)\chi^{Y'}(\pi)\right] \\
&=&  \frac{(2n-1)!!}{\prod_j j^{\pi_j}\cdot \pi_j!}
\sum_{Y,Y'} \chi^Y([2^n])\chi^{Y'}(\pi)
\left[\frac{1}{(2n)!}\sum_{\sigma\in S_{2n}}\chi^Y(\sigma)\chi^{Y'}
(\tau\sigma)\right].
\end{eqnarray*}
The slight variant
\begin{eqnarray*}
\frac{1}{(2n)!}\sum_{\sigma\in S_{2n}}\chi^Y(\sigma)\chi^{Y'}(\tau\sigma)=
\frac{\chi^Y(\tau)}{\chi^Y([1^{2n}])}\cdot \delta_{Y,Y'}
\end{eqnarray*}
of the orthogonality relations of the first kind
thus gives
\begin{eqnarray*}
\delta_{[\sigma],[2^n]}\cdot\delta_{[\tau\sigma],[\pi]} &=&
\frac{(2n-1)!!}{\prod_j j^{\pi_j}\cdot \pi_j!} \sum_{Y,Y'}
\chi^Y([2^n])\chi^{Y'}(\pi)\left[
\frac{\chi^{Y}([2n])}{\chi^{Y}([1^{2n}])}\cdot\delta_{Y,Y'}\right]\\
&=& \frac{(2n-1)!!}{\prod_j j^{\pi_j}\cdot \pi_j!}
\sum_{Y} \frac{\chi^Y([2^n])\chi^Y(\pi)\chi^Y([2n])}{
\chi^Y([1^{2n}])}
\end{eqnarray*}
since $\tau\in[2n]$ completing the proof of (\ref{firstclaim1}).

Plugging this into our expression for $P(n,N)$, we find
\begin{eqnarray}\label{E:erni}
P(n,N) & = & (2n-1)!!\cdot
\sum_{Y} \frac{\chi^Y([2^n])\chi^Y([2n])}{\chi^Y([1^{2n}])}
\underbrace{\frac{1}{(2n)!}\sum_{\pi\in S_{2n}}N^{\sum_i\pi_i}
\chi^Y(\pi)}_{(*)}
\end{eqnarray}
since
$$\aligned
\sum_{[\pi]} \frac{1}{\prod_j j^{\pi_j}\cdot
\pi_j!}N^{\sum_i\pi_i}\chi^Y(\pi)&=\sum_{\pi\in
S_{2n}}\frac{1}{|[\pi]|\left(\prod_j j^{\pi_j}\cdot
\pi_j!\right)}N^{\sum_i\pi_i}
\chi^Y(\pi)\\
&=\frac{1}{(2n)!}\sum_{\pi\in S_{2n}}N^{\sum_i\pi_i}
\chi^Y(\pi).\\
\endaligned$$

Rewriting the factor
\begin{equation*}
N^{\sum_i\pi_i}=\prod_{\pi_i}\left(\sum_{h=1}^N 1^{i}\right)^{\pi_i}
\end{equation*}
as a product of power sums
$p_i(x_1,\dots,x_N)=\sum_{h=1}^N x_h^{i}$,
we identify via the Frobenius Theorem \cite{Ram-Remmel} the term $(*)$ in eq.~(\ref{E:erni}) as a
special value of  the Schur polynomial $s_{Y}(x_1,\dots,x_N)$
of $Y$ over $N\ge 2n$ indeterminates, namely,
\begin{equation}\label{E:Schur}
s_{Y}(1,\dots,1)
=
\frac{1}{(2n)!}\sum_{\pi\in S_{2n}}  \chi^{Y}(\pi)\,
\prod_{\pi_i} p_i(1,\dots,1)^{\pi_i}
=\frac{1}{(2n)!}\sum_{\pi\in S_{2n}} \prod_{\pi_i}
\left(\sum_{h=1}^N 1^{i}\right)^{\pi_i}
\, \chi^{Y}(\pi).
\end{equation}

We use the Murnaghan-Nakayama rule \cite{Sagan}
\begin{equation}\label{E:MN}
\chi^Y( ( i_1,\dots,i_m)\, \sigma ) =
\sum_{Y_\mu; \, Y \setminus Y_\mu\, \text{\rm is a} \atop
      \text{\rm skew hook of length $m$}} (-1)^{w(Y_\mu)}
                                          \chi^{Y_\mu}(\sigma)
\end{equation}
to explicitly
compute the remaining character values,
where $w(Y_\mu)$ equals the number of rows in the skew hook minus one.
Let $Y_{p,q}$ denote a $(p,q)$-hook Young diagram having a single row of
length $q+1\ge 1$ and $p$ rows of size one, where $p+q+1=2n$.
It follows that $\chi^Y(( i_1,\dots,i_n)) = (-1)^p \delta_{Y,Y_{p.q}}$
since for $\sigma= ( i_1,\dots,i_n)$ the only skew hook of
length $2n$ is a hook of length $2n$ itself.
Since $\tau\in[2n]$, this implies that only Young diagrams $Y_{p,q}$
contribute to the sum in eq.~(\ref{E:erni}), and setting
$\chi^{p,q}=\chi^{Y_{p,q}}$, we arrive at
\begin{eqnarray}\label{E:erni-pq}
P(n,N)
& = & (2n-1)!!\cdot
\sum_{0\le p,q\atop p+q=2n-1} \frac{\chi^{p,q}([2^n])
\chi^{p,q}([2n])}{\chi^{p,q}([1^{2n}])}
\frac{1}{(2n)!}\sum_{\pi\in S_{2n}}N^{\sum_i\pi_i} \chi^{p,q}(\pi).
\end{eqnarray}

Eq.~(\ref{E:MN}) furthermore implies
\begin{equation}\label{E:second}
\chi^{p,q}(( i_1,\dots,i_m) \sigma)=
\begin{cases}
 \chi^{p,q-m}(\sigma) + (-1)^{m-1}\, \chi^{p-m,q}(\sigma); &
                                              \text{\rm for $m<2n$,}\\
      (-1)^{p};                           & \text{\rm for $m=2n$,}
\end{cases}
\end{equation}
where $\chi^{p-m,q}$ and $\chi^{p,q-m}$ are zero in the respective cases $p-m<0$ and
$q-m<0$.
This recursion allows us to prove via induction
\begin{eqnarray}
\chi^{p,q}([2^n]) &= &
\begin{cases}
(-1)^{\frac{p}{2}}\binom{n-1}{\frac{p}{2}}; & \text{\rm for $p\equiv 0
\mod 2$,}\\
(-1)^{\frac{p+1}{2}}\binom{n-1}{\frac{p-1}{2}}; & \text{\rm for $p\equiv 1
\mod 2$,}
\end{cases}\\
\chi^{p,q}([1^{2n}]) & = & \binom{2n-1}{q}.
\end{eqnarray}

Now, according to the definition of Schur polynomials, we have
\begin{equation}
s_{\lambda}(x_1,x_2,\ldots,x_n)=\sum_{T}x^{T}=\sum_{T}
x_1^{t_1}\cdots x_n^{t_n},
\end{equation}
where the summation is over all semistandard Young tableaux $T$ of
shape $\lambda$, and $t_i$ counts the occurrences of the number $i$
in $T$. Thus, $s_{p,q}(\underbrace{1,\ldots,1}_{N})$ counts the
number of semistandard Young tableaux $T$ of shape
$(q+1,\underbrace{1,\ldots,1}_{p})$ whose contents are integers
not larger than $N$, which we next compute. If the first element in the first row is $i$, where $1\leq i\leq
N-p$, then there are $\binom{N-i}{p}$ ways to arrange the numbers in the
first column. Since the first row weakly increases, the remaining $q$
elements in the first row can be chosen from $i$ to $N$ with
repetition. There are thus $\binom{N+q-i}{q}$ ways to choose  $q$
elements in the first row. We  conclude that the number of desired
semistandard Young tableaux is
\begin{equation}
s_{p,q}(1,\dots,1)=\sum_{i=1}^{N-p}\binom{N-i}{p}\binom{N+q-i}{q}
=\binom{N+q}{2n}\binom{2n-1}{q},
\end{equation}
and hence
\begin{equation}
 \frac{1}{(2n)!}
\sum_{\pi\in S_{2n}} N^{\sum_i\pi_i} \chi^{p,q}(\pi)=
\binom{N+q}{2n}\
\binom{2n-1}{q}.
\end{equation}

Consequently, we arrive at
\begin{eqnarray*}
\frac{P(n,N)}{(2n-1)!!} & =& \sum_{j=0}^{n-1}(-1)^j\binom{n-1}{j}\left[
\binom{N+2n-2j-1}{2n}+\binom{N+2n-2j-2}{2n}\right] \\
& =& \sum_{j=0}^{n-1}(-1)^j \binom{n-1}{j}\frac{1}{2\pi i} \oint
\frac{(1+x)^{N+2n-2j-1}}{x^{N-2j}}+
\frac{(1+x)^{N+2n-2j-2}}{x^{N-2j-1}}dx\\
& =& \frac{1}{2\pi i} \oint \frac{(1+x)^N}{x^N}\left( 1+2x\right)
\sum_{j=0}^{n-1}(-1)^j\binom{n-1}{j}x^{2j}(1+x)^{2n-2j-2}dx \\
 & = & \frac{1}{2\pi i} \oint \frac{(1+x)^N}{x^N}\left( 1+2x\right)^ndx \\
 &= & \frac{1}{2}\frac{1}{2\pi i}
        \oint \frac{1}{z^{n+2}}\left(\frac{1+z}{1-z}\right)^Ndz,
\end{eqnarray*}
where $z=(1+2x)^{-1}$.
We have thus proved
$2\frac{P(n,N)}{(2n-1)!!}=[z^{n+1}]\left(\frac{1+z}{1-z}\right)^N$,
or equivalently
\begin{equation*}
1+ 2\sum_{n=0}^{\infty}\frac{P(n,N)}{(2n-1)!!}z^{n+1} =
\left(\frac{1+z}{1-z}\right)^N.
\end{equation*}

For convenience, we now introduce
$$p(n,N)={{P(n,N)}\over{(2n-1)!!}},$$
and claim
that
\begin{equation}\label{E:11}
1+ 2\sum_{n=0}^{\infty}p(n,x)\, z^{n+1} =
\left(\frac{1+z}{1-z}\right)^x
\end{equation}
for an integral indeterminate $x$, i.e., we now drop the restriction
that $N\geq 2n$.
Of course, $p(n,x)$ is determined by $\{p(n,N)\mid N\ge 2n\}$ since
any two polynomials  that coincide on an arbitrarily
large set must be identical. The identity
\begin{equation*}
\left(\frac{1+z}{1-z}\right)^N=(1+2z+2z^2+2z^3+\dots)
                                        \left(\frac{1+z}{1-z}\right)^{N-1}
\end{equation*}
gives
\begin{equation*}
p(n,N)=1+p(n,N-1)+2\bigl (
p(n-1,N-1)+p(n-2,N-1)+\cdots +p(1,N-1)
\bigr ),
\end{equation*}
whence
\begin{equation*}
p(n,N)-p(n-1,N)=p(n,N-1)-p(n-1,N-1)+2p(n-1,N-1),
\end{equation*}
from which it follows that
\begin{eqnarray*}
p(n,N) & = & p(n,N-1)+p(n-1,N)+p(n-1,N-1),~{\rm for~all}~N\geq 2n,
\end{eqnarray*}
where $p(0,N)=N$ and $p(n,0)=0$, for $n,N\ge 0$.  This recursion guarantees the
identity of polynomials
\begin{eqnarray*}
    p(n,x) & = & p(n,x-1)+p(n-1,x)+p(n-1,x-1),
\end{eqnarray*}
where $p(0,x)=x$ and $p(n,0)=0$, for $n\ge 0$. At the same time,
\begin{equation*}
\left(\frac{1+z}{1-z}\right)^x = 1+ 2\sum_{n=0}^{\infty}b(n,x)\, z^{n+1}
\end{equation*}
gives $b(n,x) = b(n,x-1)+b(n-1,x)+b(n-1,x-1)$, where $b(0,x)=x$
and $b(n,0)=0$, for $n\ge 0$. Thus, the two sides of eq.~(\ref{E:11})
satisfy the same recursions and initial conditions, guaranteeing their
equality, as was claimed, completing the proof of the lemma.
\end{proof}

\begin{lemma}\label{L:recursion}
The ${\bf c}_g(n)$ satisfy the recursion
\begin{eqnarray}\label{E:recursion}
(n+1)\, \mathbf{c}_g(n) & = &  2(2n-1)\,\mathbf{c}_g(n-1)+
                          (2n-1)(n-1)(2n-3)\,\mathbf{c}_{g-1}(n-2),
\end{eqnarray}
where $\mathbf{c}_g(n)=0$ for $2g>n$.
\end{lemma}

\begin{proof}
On the one hand by definition,
$$\aligned
{\partial\over{\partial z}} \sum_{n\geq 0} p(n,x)z^{n+1}&=\sum_{n\geq 0} (n+1) p(n,x)z^n\\
&=\sum_{n\geq 0} (n+1)\sum_{2g\leq n} {{{\bf c}_g(n)x^{n+1-2g}}\over{(2n-1)!!}} z^n,\\
\endaligned$$
so the coefficient of $x^{n+1-2g}z^n$ is ${{(n+1){\bf c}_g(n)}\over{(2n-1)!!}}$.
On the other hand by eq.~(\ref{E:11}),
$$\aligned
{\partial\over{\partial z}} \sum_{n\geq 0} p(n,x)z^{n+1}&={x\over{1-z^2}} \left(\frac{1+z}{1-z}\right)^x\\
&={x\over{1-z^2}}\bigl (
1+2\sum_{n\geq 0} p(n,x)z^{n+1}
\bigr )\\
&=x\bigl (
1+\sum_{n\geq 0}\sum_{2g\leq n} {{{\bf c}_g(x)}\over{(2n-1)!!}}z^{n+1}
\bigr )~\sum_{j\geq 0} z^{2j}\\
\endaligned$$
has $2\sum_{j=0}^g {{{\bf c}_{g-j}(n-1-2j)}\over{(2(n-1-2j)-1)!!}}$
as its coefficient of $x^{n+1-2g}z^g$.  Equating these two coefficients, we obtain
\begin{eqnarray*}
                   \frac{(n+1){\bf c}_g(n)}{(2n-1)!!} =
2\sum_{j=0}^{g}\frac{{\bf c}_{g-j}(n-1-2j)}{\left(2(n-1-2j)-1\right)!!},
\end{eqnarray*}
and hence
\begin{eqnarray*}
&& \frac{(n+1){\bf c}_g(n)}{(2n-1)!!}-
   \frac{(n-1){\bf c}_{g-1}(n-2)}{(2(n-2)-1)!!}\\
&=&
2\sum_{j=0}^{g}\frac{{\bf c}_{g-j}(n-1-2j)}
{\left(2(n-1-2j)-1\right)!!}-2\sum_{j=0}^{g-1}\frac{{\bf c}_{g-1-j}(n-3-2j)}
{\left(2(n-3-2j)-1\right)!!}\\
&=&2\frac{{\bf c}_{g}(n-1)}
{\left(2(n-1)-1\right)!!},
\end{eqnarray*}
as required.
\end{proof}

\begin{theorem}\label{E:GF}
For any $g\ge 1$ the generating function ${\bf C}_g(z)=
\sum_{n\ge 0}\mathbf{c}_g(n)z^n$ is given by
\begin{eqnarray}\label{E:it}
\mathbf{C}_g(z) = \, P_g(z)\frac{\sqrt {1-4\,z}}{(1-4z)^{3g}},
\end{eqnarray}
where $P_g(z)$ is a polynomial with integral coefficients of degree at 
most $(3g-1)$, $P_g(1/4)\neq 0$, $[z^{2g}]P_g(z)\neq 0$ and 
$[z^h]P_g(z)=0$ for $0\leq h\leq 2g-1$.
In particular, ${\bf C}_g(z)$ is algebraic over $\mathbb{C}(z)$ and has
its unique singularity at $z=1/4$ independent of genus. 
Furthermore, the coefficients of ${\bf C}_g(z)$ have the asymptotics
\begin{equation}
[z^n]{\bf C}_g(z)\sim \frac{P_g(\frac{1}{4})}{\Gamma(3g-1/2)} n^{3g-\frac{3}{2}} 4^n.
\end{equation}
\end{theorem}
\begin{proof}

The recursion eq.~(\ref{E:recursion}) is equivalent to the ODE
\begin{eqnarray}\label{E:ODE}
z(1-4z)\frac{d}{dz}\mathbf{C}_g(z) +(1-2z)\mathbf{C}_g(z) & = &
\Phi_{g-1}(z),
\end{eqnarray}
where
\begin{eqnarray*}
\Phi_{g-1}(z) =  z^2\left(4z^3\frac{d^3}{dz^3}\mathbf{C}_{g-1}(z) +
24z^2 \frac{d^2}{dz^2}\mathbf{C}_{g-1}(z) + 27z\frac{d}{dz}
\mathbf{C}_{g-1}(z)+3\mathbf{C}_{g-1}(z)\right)
\end{eqnarray*}
with initial condition $\mathbf{C}_g(0)=0$
since $r=n+1-2g$ has no positive solution $r>0$ for $n<2g$.


We prove the theorem by induction on $g$. For the basis step
$g=1$,  we have the generating function ${\bf C}_0(z)=2\, \left( 1+\sqrt {1-4\,z}
\right) ^{-1}$ for the Catalan numbers.
The Picard-Lindel\"of Theorem \cite{Ince} guarantees the unique
solution to eq.~(\ref{E:ODE}), satisfying ${\bf C}_1(0)=0$,
given by
\begin{equation}
{\bf C}_1(z)={\frac {{z}^{2}}{\left( 1-4\,z \right)^{3}}}
              \sqrt {1-4\,z}.
\end{equation}

For the inductive step, the induction hypothesis gives that for any
$j\leq g$, we have
\begin{equation*}
{\bf C}_j(z)=\frac{P_j(z)}{(1-4z)^{3j}}\,\sqrt{1-4z},
\end{equation*}
where $P_j(x)$ is an integral polynomial of degree at most $3j-1$, $P_j(1/4)\neq 0$,
$[z^{2j}]P_j(z)\neq0$, and $[z^h]P_j(z)=0$ for $0\leq h\leq 2j-1$.
The general solution of eq.~(\ref{E:ODE}) is
\begin{equation}\label{E:intsoln}
{\bf C}_{g+1}(z)=\,\left(\int_0^z\frac{\Phi_{g}(y)}{(1-4y)^{3/2}}dy+C
\right)\,\frac{\sqrt{1-4z}}{z},
\end{equation}
where
\begin{eqnarray*}
\Phi_{g}(z) & = &
4z^5\frac{{d}^3 }{{d} z^3}{\bf C}_{g}(z)+
24z^4\frac{{d}^2 }{{d} z^2}{\bf C}_{g}(z)+
27z^3\frac{{d} }{{d}z}{\bf C}_{g}(z)+3z^2{\bf C}_{g}(z) \\
& = & \frac{Q_g(z)}{(1-4z)^{3g+5/2}}.
\end{eqnarray*}

We claim that\\\\
$\bullet$~$Q_g(z)$ is a polynomial of degree at most $3g+2$,\\
$\bullet$~ $Q_g(1/4)\neq 0$, and\\
$\bullet$~ $[z^{2g+2}]Q_g(z)\neq 0$ and $[z^h]Q_g(z)=0$ if $0\leq h\leq 2g+1$.\\\\
To these ends and by the inductive hypothesis ${\bf C}_{g}(z)=P_{g}(z)/(1-4z)^{3g-1/2}$, we have
$$\aligned
\frac{\mathrm{d} {\bf C}_{g}(z)}{\mathrm{d}
z}&=\frac{P_{1g}(z)}{(1-4z)^{3g+1/2}},~{\rm where}~P_{1g}(z)=(1-4z)P_{g}'(z)+(12g-2)P_{g}(z),\\
\quad \frac{\mathrm{d}^2 {\bf
C}_{g}(z)}{\mathrm{d} z^2}&=\frac{P_{2g}(z)}{(1-4z)^{3g+3/2}},~{\rm where}~P_{2g}(z)=(1-4z)P_{1g}'(z)+(12g+2)P_{1g}(z), \\
\quad
\frac{\mathrm{d}^3 {\bf C}_{g}(z)}{\mathrm{d}
z^3}&=\frac{P_{3g}(z)}{(1-4z)^{3g+5/2}},~{\rm where}~P_{3g}(z)=(1-4z)P_{2g}'(z)+(12g+6)P_{2g}(z).\\
\endaligned$$
Thus,
\begin{equation}\label{E:Q}
Q_g(z)=4z^5P_{3g}(z)+24z^4(1-4z)P_{2g}(z)
+27z^3(1-4z)^2P_{1g}(z)+3z^2(1-4z)^3P_{g}(z).
\end{equation}

To see that $Q_g(z)$ is
indeed a polynomial of degree at most $3g+2$, first note that
$P_{g}(z)$ and each $P_{ig}(z)$, for $1\leq i\leq 3$, are each polynomials of
degree at most $3g-1$, so the degree of $Q_g(z)$ is at most
$3g+4$. We shall compute the coefficients $[z^{3g+3}]Q_g(z)$
and $[z^{3g+4}]Q_g(z)$ in terms of $d_{3g-2}=[z^{3g-2}]P_{g}(z)$
and $d_{3g-1}=[z^{3g-1}]P_{g}(z)$, where
we find
\begin{eqnarray*}
&&[z^{3g-1}]P_{1g}(z)= 2d_{3g-1},\\
&&[z^{3g-2}]P_{1g}(z)=(3g-1)d_{3g-1}+6d_{3g-2},\\\\
&&[z^{3g-1}]P_{2g}(z)=6[z^{3g-1}]P_{1g}(z)=12d_{3g-1}, \\
&&[z^{3g-2}]P_{2g}(z)=(3g-1)[z^{3g-1}]P_{1g}(z)
+10[z^{3g-2}]P_{1g}(z)\\
&&\qquad\qquad\qquad\ \hskip -.75ex=12(3g-1)d_{3g-1}+60d_{3g-2},\\\\
&&[z^{3g-1}]P_{3g}(z)=10[z^{3g-1}]P_{2g}(z)=120d_{3g-1},
\\ &&[z^{3g-2}]P_{3g}(z)=(3g-1)[z^{3g-1}]P_{2g}(z)
+14[z^{3g-2}]P_{2g}(z)\\
&&\qquad\qquad\qquad\ \hskip -.75ex=180(3g-1)d_{3g-1}+840d_{3g-2},\\
\end{eqnarray*}
Plugging these into eq.~(\ref{E:Q}), we compute
\begin{eqnarray*}
[z^{3g+4}]Q_g(z)&=&4[z^{3g-1}]P_{3g}(z)-24\times4\times[z^{3g-1}]
P_{2g}(z)\\
&&+27\times(-4)^2\times[z^{3g-1}]P_{1g}(z)+3\times(-4)^3\times
d_{3g-1}\\
&&\hskip -4ex =0,
\end{eqnarray*}
and
\begin{eqnarray*}
[z^{3g+3}]Q_g(z)&=&4[z^{3g-2}]P_{3g}(z)+24[z^{3g-1}]P_{2g}(z)
-24\times4\times[z^{3g-2}]
P_{2g}(z)\\
&&+27\times(-8)\times[z^{3g-1}]P_{1g}(z)
+27\times(-4)^2\times[z^{3g-2}]P_{1g}(z)\\
&&+3\times3\times(-4)^2\times
d_{3g-1}+3\times(-4)^3\times
d_{3g-2}\\
&=&4\times\left(180(3g-1)\,d_{3g-1}+840\,d_{3g-2}\right)
+24\times12\,d_{3g-1}\\
&&-96\times\left(12(3g-1)\,d_{3g-1}+60\,d_{3g-2}\right)\\
&&-27\times16\,d_{3g-1}+27\times16\times
\left((3g-1)\,d_{3g-1}+6\,d_{3g-2}\right)\\
&&+144\,d_{3g-1}-192\,d_{3g-2}\\
&=&0.
\end{eqnarray*}

Furthermore, insofar as $P_{g}(1/4)\neq 0$, we have $P_{ig}(1/4)\neq 0$, for $1\leq i\leq 3$,
and hence also $Q_g(1/4)\neq 0$ from eq.~(\ref{E:Q}) as was claimed.

We finally show
\begin{equation*}
[z^{2g+2}] Q_g(z)  \neq  0 \quad \text{\rm and}\quad
[{z^{h}}]      Q_g(z)  =     0, \quad\text{\rm for $0\leq h\leq 2g+1$}.
\end{equation*}
By the induction hypothesis, we have
$[z^{2g}]P_{g}(z)\neq 0$ and $[z^h]P_{g}(z)=0,$
for $0\leq h<2g$.
By definition of $P_{1g},P_{2g},P_{3g}$, we have, for $g\geq 2$,
\begin{eqnarray*}
&&[z^{2g-1}]P_{1g}(z)=2g[z^{2g}]P_{g}(z),\\
&&[z^{2g-2}]P_{2g}(z)=(2g-1)[z^{2g-1}]P_{1g}(z)= (2g-1)2g[z^{2g}]P_{g}(z)\\
&&[z^{2g-3}]P_{3g}(z)=(2g-2)[z^{2g-2}]P_{2g}(z)=(2g-2)(2g-1)2g[z^{2g}]P_{g}(z),
\end{eqnarray*}
and consequently conclude
\begin{eqnarray*}
[z^{2g+2}]Q_g(z) &=& 4[z^{2g-3}]P_{3g}(z)+24[z^{2g-2}]P_{2g}(z)
 +27[z^{2g-1}]P_{1g}(z)+3[z^{2g}]P_{g}(z)\\
 &\neq& 0
\end{eqnarray*}
as was asserted.

We proceed by extracting the following coefficients:
\begin{eqnarray*}
&&[z^h]P_{1g}(z)=(h+1)[z^{h+1}]P_{g}(z)-4h[z^{h}]P_{g}(z)
+(12g-2)[z^h]P_{g}(z)\\
&&\hskip 1.5ex\qquad\qquad=0,{~\rm for}~0\leq h<2g-1,\\
&&[z^h]P_{2g}(z)=(h+1)[z^{h+1}]P_{1g}(z)
-4h[z^{h}]P_{1g}(z)+(12g+2)[z^h]P_{1g}(z)\\
&&\hskip 1.5ex\qquad\qquad=0,{~\rm for}~0\leq h<2g-2,\\
&&[z^h]P_{3g}(z)=(h+1)[z^{h+1}]P_{2g}(z)-4h[z^{h}]P_{2g}(z)
+(12g+6)[z^h]P_{2g}(z)\\
&&\hskip 1.5ex\qquad\qquad=0,{~\rm for}~0\leq h<2g-3.
\end{eqnarray*}
For $0\leq h\leq 2g+1$, we therefore conclude
\begin{eqnarray*} [z^h]Q_g(z)&=&4[z^{h-5}]P_{3g}(z)+24[z^{h-4}]P_{2g}(z)
-96[z^{h-5}]P_{2g}(z)\\
&&+27[z^{h-3}]P_{1g}(z)-27\times8[z^{h-4}]P_{1g}(z)
+27\times16[z^{h-5}]P_{1g}(z)\\
&&+3[z^{h-2}]P_{g}(z)-36[z^{h-3}]P_{g}(z)+144[z^{h-4}]P_{g}(z)
-192[z^{h-5}]P_{g}(z)\\
&=&0\\
\end{eqnarray*}
as required, completing the verifications of our assertions
about $Q_g$.

Now, since $Q_g(z)$ is a polynomial of degree at most $3g+2$, the partial fraction expansion is
given by
\begin{equation}\label{E:pfrac}
\frac{Q_g(z)}{\left( 1-4\,z \right) ^{(3g+4)}} =
               \sum_{j= 2}^{3g+4}\frac{A_j}{\left( 1-4\,z \right) ^{j}},
\end{equation}
where the $A_j\in \mathbb{Q}$. In light of
eq.~(\ref{E:intsoln}), we then compute
\begin{eqnarray*}
{\bf C}_{g+1}(z) & = &
\left(\sum_{j= 2}^{3g+4}\frac{ A_j}{4(j-1)\left( 1-4\,z \right) ^{j-1}}+
 C \right) \frac{\sqrt {1-4\,z}}{z} \\
&=& \frac{1}{z}\left(\sum_{j= 2}^{3g+4}\frac{ A_j}{4(j-1)}
\left( 1-4\,z \right)^{3g+4-j}+C\left( 1-4\,z \right)^{3g+3}\right)
\frac{\sqrt{1-4\,z }}{\left(1-4\,z \right)^{3g+3}},
\end{eqnarray*}
where the initial condition ${\bf C}_{g+1}(0)=0$ evidently guarantees
$C=\sum_{j= 2}^{3g+4} \frac{- A_j}{4(j-1)}$.

Introducing the Laurent polynomial
\begin{equation}\label{E:laurent}
P_{g+1}(z)= {{-1}\over 4z}
\left(\sum_{j= 2}^{3g+4}\frac{-A_j}{j-1}\left( 1-4\,z \right)^{3g+4-j}+
\sum_{j= 2}^{3g+4}\frac{A_j}{j-1}\left( 1-4\,z \right)^{3g+3}\right),
\end{equation}
we claim that\\\\
$\bullet$~$P_{g+1}(z)$ is a polynomial,  i.e., $[z^{-1}]P_{g+1}(z)=0$,  of 
degree at most $3g+2$,\\
$\bullet$~ $P_{g+1}(1/4)\neq 0$, and\\
$\bullet$~ $[z^{2g+2}]P_{g+1}(z)\neq 0$ and $[z^h]P_{g+1}(z)=0$ if $0\leq h\leq 2g+1$.\\\\
To these ends, we have
\begin{equation}
(-4)^{-s}[z^s]P_{g+1}(z)=
\sum_{j= 2}^{3g+4}\frac{-A_j}{j-1}\binom{3g+4-j}{s+1}+\sum_{j= 2}^{3g+4}
\frac{A_j}{j-1}\binom{3g+3}{s+1}.
\end{equation}
In particular, $[z^{-1}]P_{g+1}(z) =\sum_{j= 2}^{3g+4}\frac{ A_j}{ 4(j-1)}+
\sum_{j= 2}^{3g+4}\frac{- A_j}{4(j-1)}=0$, whence $P_{g+1}(z)$ is indeed a
polynomial of degree at most $3g+2$.  
Furthermore, we find
$P_{g+1}(1/4)=A_{3g+4}/(3g+3)\neq 0$ directly from eq.~(\ref{E:laurent}).

It remains to show by
induction on $s$ that $[z^s]P_{g+1}(z)=0$, for $0\leq s\leq 2g+1$
and $[z^{2g+2}]P_{g+1}(z)\neq 0$.
We have seen that the coefficients
$[z^{2g+2}]Q_g(z)\neq 0$ and $[z^h]Q_g(z)=0$, for $0\leq h\leq 2g+1$,
that is, from eq.~(\ref{E:pfrac}),
\begin{equation}\label{E:Aj1}
(-4)^{2g+2}~\sum_{j=2}^{g+2}\binom{3g+4-j}{2g+2}A_j\neq 0.
\end{equation}
and
\begin{equation}\label{E:Aj2}
(-4)^h~\sum_{j=2}^{3g+4-h}\binom{3g+4-j}{h}A_j=0, ~{\rm for}~ 0\leq h\leq
2g+1.
\end{equation}

It follows from eq.~(\ref{E:Aj2}) that
\begin{eqnarray*}
[z^{0}]P_{g+1}(z)&=&\sum_{j=
2}^{3g+4}\frac{-A_j}{j-1}(3g+4-j)+\sum_{j=
2}^{3g+4}\frac{A_j}{j-1}(3g+3)\\
&=&\sum_{j= 2}^{3g+4}A_j\\
&=&0,
\end{eqnarray*}
and we hence assume by induction that
$[z^s]P_{g+1}(z)=0$ for $0\leq s< 2g+1$. To compute
 $[z^{s+1}]P_{g+1}(z)$, we thus assume
  \begin{equation}\label{E:n0}
(-4)^{-s}~[z^{s}]P_{g+1}(z)=\sum_{j=
2}^{3g+4}\frac{-A_j}{j-1}\binom{3g+4-j}{s+1}+\sum_{j=
2}^{3g+4}\frac{A_j}{j-1}\binom{3g+3}{s+1}=0,
\end{equation}
i.e., $\sum_{j=
2}^{3g+4}\frac{A_j}{j-1}\binom{3g+3}{s+1}=\sum_{j=
2}^{3g+4}\frac{A_j}{j-1}\binom{3g+4-j}{s+1}$,
and compute
\begin{eqnarray*}
(-4)^{-(s+1)}~[z^{s+1}]P_{g+1}(z)&=&\sum_{j=
2}^{3g+4}\frac{-A_j}{j-1}\binom{3g+4-j}{s+2}
+\sum_{j= 2}^{3g+4}\frac{A_j}{j-1}\binom{3g+3}{s+2}\\
&=&\sum_{j=
2}^{3g+4}\frac{-A_j}{j-1}\binom{3g+4-j}{s+1}\biggl (\frac{3g+3-j-s}{s+2}\biggr )\\
&+&\sum_{j=
2}^{3g+4}\frac{A_j}{j-1}\binom{3g+3}{s+1}\biggl (\frac{3g+2-s}{s+2}\biggr ),
\end{eqnarray*}
so that
\begin{equation}
(-4)^{-(s+1)}~[z^{s+1}]P_{g+1}(z)=\sum_{j= 2}^{3g+4}{{A_j}\over {s+2}}\binom{3g+4-j}{s+1}=0,
\end{equation}
according to eq.~(\ref{E:Aj2}), completing the inductive proof that indeed the coefficients $[z^s]P_{g+1}(z)=0$
vanish, for $0\leq s\leq
2g+1$.

Similarly, using
\begin{equation}
(-4)^{-(2g+1)}~[z^{2g+1}]P_{g+1}(z)=\sum_{j=
2}^{3g+4}\frac{-A_j}{j-1}\binom{3g+4-j}{2g+2}+\sum_{j=
2}^{3g+4}\frac{A_j}{j-1}\binom{3g+3}{2g+2}=0
\end{equation}
and eq.~(\ref{E:Aj1}), it follows that the coefficient
\begin{eqnarray*}
(-4)^{-(2g+2)}[z^{2g+2}]P_{g+1}(z)&=&\sum_{j= 2}^{3g+4}\frac{-A_j}{j-1}\binom{3g+4-j}{2g+3}
+\sum_{j= 2}^{3g+4}\frac{A_j}{j-1}\binom{3g+3}{2g+3}\\
&=&\sum_{j= 2}^{3g+4}\frac{-A_j}{j-1}\binom{3g+4-j}{2g+2}
\biggl (\frac{g+2-j}{2g+3}\biggr )\\
&+&\sum_{j= 2}^{3g+4}\frac{A_j}{j-1}\binom{3g+3}{2g+2}\biggl (\frac{g+1}{2g+3}\biggr )\\
&=&\sum_{j= 2}^{3g+4}{{A_j}\over{2g+3}}\binom{3g+4-j}{2g+2}\neq 0
\end{eqnarray*}
completing the verification of our claims about $P_{g+1}$.

To complete the proof, we must finally show that $P_g(z)$ has integral coefficients, and this follows immediately
from the expression
$$
P_g(z)~=~{\bf C}_g(z) \bigl ( \sqrt{1-4z}\bigr )^{6g-1}~=~{\bf C}_g(z)\biggl(1-2z{\bf C}_0(z) \biggr )^{6g-1}
$$
since the right-hand side evidently has all integral coefficients.

\end{proof}

\begin{corollary}\label{C:BGB} We have the explicit expressions
$$\aligned
\mathbf{c}_1(n) & =  \frac{2^{n-2}(2n-1)!!}{3(n-2)!},~{\rm for}~n\geq 2,\\
\mathbf{c}_2(n) & =  \frac{2^{n-4}(5n-2)(2n-1)!!}{90 (n-4)!},~{\rm for}~n\geq 4,\\
\mathbf{c}_3(n) & =\frac{2^{n-6}(35n^2-77n+12)(2n-1)!! }{5670 (n-6)! },~{\rm for}~n \geq 6.
\endaligned$$
\end{corollary}

\begin{proof}
Differentiating the formula for $\mathbf{C}_1(z)$ using the explicit expression
for $P_1(z)$ in the Introduction, we get
$(1-4z)z \mathbf{C}_1(z) -(2+2z) \mathbf{C}_1(z) = 0$ giving the recursion
$(n-2)\mathbf{c}_1(n) = 2(2n-1) \mathbf{c}_1(n-1) $, whose solution is given by
the asserted expression for ${\bf c}_1(n)$.  As to  ${\bf c}_2(n)$ and again using
the explicit expression for $P_2(z)$ in the Introduction,
observe that ${\bf C}_2(z)$ satisfies the ODE
$$5z^2(4z-1) \frac{d^2 \mathbf{C}_2(z)}{dz^2}+ 2z(21z+11)
\frac{d\mathbf{C}_2(z)}{dz}+ 2(3z-14) \mathbf{C}_2(z) = 0$$
giving the recursion $(5n^2-27n+28)\mathbf{c}_2(n) = (20n^2-18n+4) \mathbf{c}_2(n-1)$,
which gives the asserted formula. For genus $3$ we similarly get the ODE
\begin{equation*}
(140z^4-35z^3) \frac{d^3 \mathbf{C}_3(z)}{dz^3}+
(462z^3-252z^2)\frac{{d}^2 \mathbf{C}_3(z)}{dz^2}+
(48z^2-684z)\frac{d\mathbf{C}_3(z)}{ d z}- (60z-744)
\mathbf{C}_3(z) =0
\end{equation*}
which gives the recursion
$$
(35n^3-357n^2+1006n-744) \mathbf{c}_3(n)=(140n^3-378n^2+202n-24)
\mathbf{c}_3(n-1)
$$
solved by the stated formula for $\mathbf{c}_3(n)$.
\end{proof}

\begin{corollary}\label{C:CG2G}
We have the explicit expression
$${\bf c}_g(2g)=\frac{(4g)!}{4^g (2g+1)!}.$$
Furthermore, the exponential generating function 
of these numbers is given by
$$ \sum_{g=0}{\bf c}_g(2g) \frac{x^{2g}}{(2g)!}= \frac{\sqrt{1+2x}-\sqrt{1-2x}}{2x}.$$
\end{corollary}

\begin{proof}
According to the recursion eq.~(\ref{E:recursion}), we have
$$(2g+1){\bf c}_g(2g)=(2g-1)(4g-1)(4g-3) {\bf c}_{g-1}(2g-2)$$
since ${\bf c}_g(2g-1)=0$ from Theorem \ref{E:GF}.  The asserted formula
is the unique solution of this recursion with initial value ${\bf c}_1(2)=1$.
For the second assertion, we compute
\begin{eqnarray*}
\sum_{g=0}a_g \frac{x^{2g}}{(2g)!}       
&=& \sum_{g=0} \frac{(4g)!}{(2g)!(2g+1)!} \left(\frac{x}{2} \right)^{2g}\\
            &=& \sum_{g=0} C_{2g} \left(\frac{x}{2} \right)^{2g}\\
            &=& \frac{{\bf C}(\frac{x}{2})+{\bf C}(-\frac{x}{2})}{2}\\
            &=& \frac{\sqrt{1+2x}-\sqrt{1-2x}}{2x}
\end{eqnarray*}
as was claimed.
\end{proof}

The formulas for ${\bf c}_g(2g)$ and ${\bf c}_1(n)$ were given in 
\cite{Cori-Marcus}.  In fact and as illustrated in the proofs,
our methods provide a framework for computing
explicit expressions of $\mathbf{c}_g(n)$, for any $g$.

\section{Macromolecular diagrams of genus $g$}

We extend the enumerative results of Theorem \ref{E:GF} to macromolecular diagrams by first
specializing to shapes and then
modifying shapes to produce macromolecular diagrams.

\begin{lemma}\label{L:GFbi-sh}
Suppose $g$ is a non-negative integer. Then
\begin{eqnarray}\label{E:GFsh}
{\bf S}_g(z,u) & = & \frac{1+z}{1+2z-zu}
                    {\bf C}_g\left(\frac{z(1+z)}{(1+2z-zu)^2}\right).
\end{eqnarray}
\end{lemma}
\begin{proof}
We first prove
\begin{eqnarray}\label{E:GFbi}
{\bf C}_g(x,y) & = & \frac{1}{x+1-yx}~
                          {\bf C}_g\left(\frac{x}{(x+1-yx)^2}\right)
\end{eqnarray}
and to this end, choose $\xi\in \mathscr{C}_g(s+1,m+1)$ and label one of its
$1$-chords. Since we can label any of the $(m+1)$ $1$-chords of $\xi$,
$(m+1)~\mathbf{c}_g(s+1,m+1)$ different such labeled linear chord diagrams arise.
On the other hand, to produce $\xi$ with this labeling, we can add one labeled $1$-chord to an element of
$\mathscr{C}_g(s,m+1)$ by inserting a parallel copy of an existing $1$-chord or by
inserting a new
labeled $1$-chord in an element of $\mathscr{C}_g(s,m)$, where
we may only insert the $1$-chord between two vertices not
already forming a $1$-chord. It follows that we have the recursion
\begin{equation*}\label{E:regk}
(m+1)~\mathbf{c}_g(n+1,m+1)=
                        (m+1)~\mathbf{c}_g(n,m+1)+(2n+1-m)~\mathbf{c}_g(n,m)
\end{equation*}
or equivalently the PDE
\begin{equation}\label{E:diffgf}
\frac{\partial {\bf C}_g(x,y)}{\partial y}=x\frac{\partial {\bf
C}_g(x,y)}{\partial y}+2x^2\frac{\partial {\bf C}_g(x,y)}{\partial
x}+x{\bf C}_g(x,y)-xy\frac{\partial {\bf C}_g(x,y)}{\partial y},
\end{equation}
which is thus satisfied by ${\bf C}_g(x,y)$.

On the other hand,
\begin{equation*}
{\bf C}^*_g(x,y)=\frac{1}{x+1-yx}{\bf C}_g\left(\frac{x}{(x+1-yx)^2}\right)
\end{equation*}
is also a solution of eq.~(\ref{E:diffgf}), which
specializes to
${\bf C}_g(x)={\bf C}^*_g(x,1)$, and moreover, we have
$\mathbf{c}^*_g(n,m)=[x^ny^m]{\bf C}^*_g(x,y) = 0$, for $m>n$.
Indeed, the first assertion is easily verified directly, the specialization is obvious,
and the fact that $y$ only appears in the
power series ${\bf C}^*_g(x,y)$ in the form of products $xy$
implies that $\mathbf{c}^*_g(n,m) =  0$, for $m>n$.
Thus, the coefficients $\mathbf c_g^*(n,m)$ satisfy the same recursion and initial conditions
as $\mathbf c_g(n,m)$, and hence by induction on $n$, we conclude $\mathbf c_g^*(n,m)=\mathbf c_g(n,m)$,
for $n,m\geq 0$.  This
proves that ${\bf C}_{g}(n,m)$ indeed satisfies eq.~(\ref{E:GFbi}) as was claimed.

To complete the proof of eq.~(\ref{E:GFsh}), we use that
the projection $\vartheta$ is surjective and affects neither the genus nor the number
of $1$-chords, namely,
\begin{equation*}
{\bf C}_g(x,y)= \sum_{m\geq 0} \sum_{{\gamma~\text{having genus $g$}} \atop \text{and $m$ $1$-chords}}
                {\bf C}_\gamma(x,y).
\end{equation*}
Furthermore, if a shape $\gamma$ has $s$ chords, of which $t$ are $1$-chords, then we evidently have
\begin{eqnarray*}
{\bf C}_\gamma(x,y) & = & \left(\frac{x}{1-x}\right)^s y^{t},
\end{eqnarray*}
which shows that ${\bf C}_\gamma(x,y)$ depends only on the total number of
chords and number of $1$-chords in $\gamma$. Consequently,
\begin{equation}\label{E:kkk1}
{\bf C}_g(x,y)= \sum_{m\geq 0} \sum_{{\gamma~\text{having genus $g$}} \atop \text{and $m$ $1$-chords}}
{\bf C}_\gamma(x,y) =\sum_{s\geq 0}\sum_{m=0}^{s}
\mathbf{s}_g(s,m)\left(\frac{x}{1-x}\right)^sy^m
={\bf S}_g\left(\frac{x}{1-x},y\right).
\end{equation}
Setting $z=\frac{x}{1-x}$, i.e., $x=\frac{z}{1+z}$, and $u=y$, we arrive at
\begin{equation*}
{\bf S}_g(z,u)=\frac{1+z}{1+2z-zu}{\bf C}_g
\left(\frac{z(1+z)}{(1+2z-zu)^2}\right),
\end{equation*}
as required.
\end{proof}


\begin{lemma} \label{L:GFbull}For any shape $\gamma$ with $s\geq 1$ chords and
$m\geq 0$ 1-chords, we have
$${\bf D}_{\gamma,\sigma}(z)=(1-z)^{-1}\left(\frac{z^{2\sigma}}{(1-z^2)(1-z)^2-(2z-z^2)
z^{2\sigma}}\right)^s \, z^m.$$
In particular, ${\bf D}_{\gamma,\sigma} (z)$ depends only upon the number of
chords and 1-chords in $\gamma$.
\end{lemma}

\begin{proof}
We shall construct $\sqcup_{n\geq 0}{\mathcal D}_{\gamma,\sigma}(n)$ with simple combinatorial building blocks.  As a point of notation and as usual, if $\mathcal X=\sqcup_{n\geq 0}\mathcal X(n)$ is a collection of sets of partial matchings on $n\geq 0$ vertices, then we consider the corresponding generating function $\mathbf X(z)=\sum_{n\geq 0} {\bf x}(n) z^n$.
In particular, we have the set $\mathcal Z$ consisting of a single vertex with generating
function ${\bf Z}(z)=z$ and the set $\mathcal R$ consisting of a single arc and no additional vertices
with generating function ${\bf R}(z)=z^2$.

Let $=$ denote set-theoretic bijection, $+$ disjoint union, $\times$ Cartesian product with iteration written as exponentiation,
$\mathcal I$  the empty set, and $ \textsc{Seq}({\mathcal X})=\mathcal I+\mathcal X+\mathcal X^2+\cdots$, for any collection $\mathcal X$.

Define the set $\mathcal L=\textsc {Seq}(\mathcal Z)$ consisting
of any number $n\geq 0$ of isolated vertices and no chords, with its generating function ${\bf L}(z)=1/(1-z)$,
and the set  $\mathcal K^\sigma$ comprised of a single stack with at least $\sigma\geq 1$ arcs and no additional vertices,
with its generating function ${\bf K}^\sigma(z)=z^{2\sigma}/(1-z)$.

The collection
$\mathcal{N}^{\sigma}=\mathcal{K}^{\sigma}\times
\left(\mathcal{Z}\times\mathcal{L}
+\mathcal{Z}\times\mathcal{L}+\left(\mathcal{Z}\times
\mathcal{L}\right)^2\right)$ of all single stacks together with a non-empty interval
of unpaired vertices on at least one side thus has generating function
\begin{equation*}
\mathbf{N}^\sigma(z)=\frac{z^{2\sigma}}{1-z^2}\left(2\frac{z}{1-z}
+\left(\frac{z}{1-z}\right)^2\right).
\end{equation*}
Furthermore, the collection $\mathcal M^\sigma= \mathcal K^\sigma\times \textsc{Seq}(\mathcal N^\sigma)$ of all pairs consisting of a stack $\mathcal K^\sigma$ and a (possibly empty) sequence of
neighboring stacks likewise has generating function
\begin{eqnarray*}
\mathbf{M}^\sigma(z)=\frac{\mathbf{K}^\sigma(z)}{1-\mathbf{N}^\sigma(z)}=
\frac{\frac{z^{2\sigma}}{1-z^2}}
{1-\frac{z^{2\sigma}}{1-z^2}\left(2\frac{z}{1-z}
+\left(\frac{z}{1-z}\right)^2\right)},
\end{eqnarray*}
where only intervals of isolated vertices as are necessary to separate the neighboring
stacks have been inserted in $\mathcal M^\sigma$.

To complete the construction and count, we must still insert possible unpaired vertices at the remaining
$2s+1$ possible locations, where there must be a non-trivial such insertion between the endpoints of each 1-chord.  These insertions correspond to ${\mathcal L}^{2s+1-m}\times (\mathcal Z\times \mathcal L)^m$,
and we therefore conclude that $\sqcup_{n\geq 0}{\mathcal P}_\gamma(n)=\left(\mathcal{M}^\sigma\right)^s\times\mathcal{L}^{2s+1-m}\times
\left(\mathcal{Z}\times \mathcal{L}\right)^{m}$ has the asserted generating function
\begin{eqnarray*}
{\bf D}_{\gamma,\sigma}(z) & = &
\left(\frac{\frac{z^{2\sigma}}{1-z^2}}
{1-\frac{z^{2\sigma}}{1-z^2}\left(2\frac{z}{1-z}
+\left(\frac{z}{1-z}\right)^2\right)}\right)^s
\left(\frac{1}{1-z}\right)^{2s+1-m}
\left(\frac{z}{1-z}\right)^{m}\nonumber\\
&=&(1-z)^{-1}\left(\frac{z^{2\sigma}}{(1-z^2)(1-z)^2-(2z-z^2)
z^{2\sigma}}\right)^s \, z^m.
\end{eqnarray*}
\end{proof}

Our main result about macromolecular diagrams follows.

\begin{theorem}\label{T:genus}
Suppose $g$ and $\sigma$ are positive natural numbers and let  $u_\sigma(z)
=\frac{(z^2)^{\sigma-1}}{z^{2\sigma}-z^2+1}$. Then the generating function ${\bf D}_{g,\sigma}(z)$
is algebraic over $\mathbb{C}(x)$ and given by
\begin{eqnarray}\label{E:oho}
{\bf D}_{g,\sigma}(z) & = & \frac{1}{u_\sigma(z) z^2-z+1}\
                            {\bf C}_g\left(\frac{u_\sigma(z)z^2}
                            {\left(u_\sigma(z) z^{2}-z+1\right)^2}\right).
\end{eqnarray}
In particular, for arbitrary but fixed $g$ and $\gamma_2\approx 1.9685$, we have
\begin{equation}
[z^n]{\bf D}_{g,2}(z)\sim k_g\,n^{3(g-\frac{1}{2})} \gamma_2^n,
\end{equation}
for some constant $k_g$.
\end{theorem}
\begin{proof}
Since each element $\mathscr{D}_{g,\sigma}(n)$ projects to a unique shape
$\gamma$ with genus $g$ and some number $m\geq 0$ of $1$-chords, we have
\begin{equation}\label{E:Hgf}
{\bf D}_{g,\sigma}(z) =
\sum_{m\geq 0}\sum_{\gamma\,  \text{\rm having genus $g$}
\atop  \text{\rm and $m$ $1$-chords}}
\mathbf{D}_{\gamma,\sigma}(z).
\end{equation}
According to Lemma \ref {L:GFbull}, ${\bf D}_{\gamma,\sigma}(z)$ only depends on
the number of chords and $1$-chords of $\gamma$, and we can therefore express
\begin{eqnarray*}
{\bf D}_{g,\sigma}(z)
 & = & {1\over{z-1}}~{\bf S}_g\biggl({{z^{2g}}\over{(1-z^2)(1-z)^2-(2z-z^2)z^{2\sigma}}},z\biggr)\\
 & = &\frac{1}{(1-z) +{u_\sigma(z)}z^2}\,
{\bf C}_g\left(\frac{z^2\, {u_\sigma(z)}}
{\bigl((1-z) +{u_\sigma(z)}z^2
\bigr)^2}\right)
\end{eqnarray*}
using Lemma \ref{L:GFbi-sh} in order to confirm eq.~(\ref{E:oho}),
where the second equality follows from direct computation.
Let
\begin{equation*}\label{E:formula}
\theta_\sigma(z)=\frac{z^2\, {u_\sigma(z)}}
{\bigl((1-z) +{u_\sigma(z)}z^2
\bigr)^2}
\end{equation*}
denote the argument of ${\bf C}_g$ in this expression

Since any algebraic
function is in particular $D$-finite as well as $\Delta$-analytic
\cite{Stanley}, we conclude from Theorem \ref{E:GF} that
\begin{equation}\label{E:sing}
\mathbf{C}_g(z)=x_g\, (1-4z)^{-(3g-1/2)}(1+o(1))
\quad\text{for }z\rightarrow 1/4,
\end{equation}
for some constant $x_g$.  Since ${\bf C}_g(z)$ is algebraic
over $K=\mathbb C(z)$, there exist polynomials $R_i(z)$, for $i=1,\ldots ,\ell$,
such that $\sum_{i=1}^\ell R_i(z)\, {\bf C}_g(z)^i=0$, whence
$\sum_{i=1}^\ell R_i(\theta_\sigma(z))\, {\bf C}_g(\theta_\sigma(z))^i=0$ as well.
Setting $L=\mathbb{C}(\theta_{\sigma}(z))$, we thus have
\begin{equation*}
[L({\bf C}_g(\theta_\sigma(z)))\ \colon \ K]=
[L({\bf C}_g(\theta_\sigma(z)))\ \colon \  L]\cdot
[L \ \colon \  K]<\infty,
\end{equation*}
i.e.,~${\bf D}_{g,\sigma}(z)$ is algebraic over $K$.
Pringsheim's Theorem \cite{Flajolet-Sedgewick}
guarantees that
for any $\sigma\ge 1$, ${\bf D}_{g,\sigma}(z)$ has a dominant real
singularity $\gamma_\sigma>0$.

In particular, for $\sigma=2$, we verify directly
that $\gamma_2$ is the unique solution of minimum modulus of
$\theta_2(z)=1/4$, which is strictly smaller than any other singularities of
$\theta_2(z)$ and satisfies $\theta'(\gamma_2)\neq 0$. It follows that
${\bf D}_{g,2}(z)$ is governed by the supercritical paradigm 
\cite{Flajolet-Sedgewick},
and hence
$\mathbf{D}_{g,2}(z)$ has the singular expansion
\begin{equation}\label{E:singprime}
\mathbf{D}_{g,2}(z)=k'_g\, (1-\gamma_2)^{-(3g-1/2)}(1+o(1))
\quad\text{for }z\rightarrow \gamma_2,
\end{equation}
for some constant
$k'_g$.

For arbitrary but fixed $g$, we thus find the asymptotics
\begin{equation}
[z^n]{\bf D}_{g,2}(z)\sim k_g \, n^{3(g-1/2)}\, \gamma_2^n,
\end{equation}
where $\gamma_2\approx 1.9685$ as was claimed.
\end{proof}

\section{Acknowledgements.}
JEA and RCP are supported by QGM (Centre for Quantum
Geometry of Moduli Spaces, funded by the Danish National Research Foundation).
CMR is supported by the 973 Project, the PCSIRT Project of
the Ministry of Education, the Ministry of Science and Technology,
and the National Science Foundation of China.

\vfill\eject

\bibliographystyle{amsplain}

\end{document}